\documentclass[11pt]{amsart}
\usepackage{mathtools}
\usepackage{amssymb}
\usepackage[margin=1in]{geometry}
\usepackage[utf8]{inputenc}
\usepackage{mathtools}
\usepackage{cancel}
\usepackage[pdfdisplaydoctitle=true,
            colorlinks=true,
            urlcolor=blue,
            citecolor=blue,
            linkcolor=blue,
            pdfstartview=FitH,
            pdfpagemode= UseNone,
            bookmarksnumbered=true]{hyperref}
\newtheorem{theorem}{Theorem}

\newtheorem{lemma}[theorem]{Lemma}
\newtheorem{proposition}[theorem]{Proposition}
\theoremstyle{definition}

\theoremstyle{remark}
\newtheorem{remark}[theorem]{Remark}

\newcommand{\RR}{\mathbb{R}}

\newcommand{\norm}[1]{\left\Vert#1\right\Vert}

\mathtoolsset{showonlyrefs}
\begin{document}

\title{Global Lagrangian Solutions of the Camassa-Holm equation}
\author[J. Lee]{Jae Min Lee}
\address{Department of Mathematics, The Graduate Center, City University of New York, NY 11106, USA}
\email{jlee10@gradcenter.cuny.edu}

\date{\today}

\maketitle 

\begin{abstract}
In this paper we construct global weak conservative solutions of the Camassa-Holm equation on the periodic domain. We first express the equation in Lagrangian flow variable $\eta$ and then transform it using a change of variable $\rho=\sqrt{\eta_x}$. The new variable removes the singularity of the CH equation, and we obtain both the global weak conservative solution and global spatial smoothness of the Lagrangian trajectories of the CH equation. This work is motivated by J. Lenells who proved similar results for the Hunter-Saxton equation using the geometric interpretation.
\end{abstract}

\section{Introduction}

The Camassa-Holm(CH) equation
$$u_t-u_{txx}+3uu_x-uu_{xxx}-2u_x u_{xx}=0,\;\;t\ \in \RR,\;x \in S^1.$$
was originally derived as a model for shallow water waves and has remarkable properties like infinitely many conservation laws via bi-hamiltonian structures \cite{AK1999, CH1993, Con2, FF1981} and soliton-like solutions \cite{BSS2000, ConM1999, ConStra, ConStra2}. Also, the solutions of this equation can be interpreted as geodesics of the right invariant Sobolev $H^1$ metric on the diffeomorphism group on the circle \cite{EM1970, KM2003, K2007, Mis1998}. In fact, in terms of the diffeomorphism $\eta$ defined by the flow equation
\begin{equation}\label{flow}
\frac{\partial \eta}{\partial t}=u(t,\eta(t,x)),
\end{equation}
the CH equation is rewritten as an ODE on the Banach space and the local well-posedness of the corresponding Cauchy problem is obtained by using the standard ODE technique.

It is known that some solution $u(t,x)$ of the CH equation develops a finite time singularity \cite{M1, Con, ConE2, ConE3}. The mechanism for this breakdown is called a wave breaking; the solution remains bounded  but its slope becomes unbounded at the breakdown time. Wave breaking can be nicely illustrated in terms of peakon-antipeakon interaction. When these two waves collide at some time, the combined wave forms an infinite slope. After this collision, there are two possible scenarios; either two waves pass through each other with total energy preserved, or annihilate each other with a lose of energy. The solutions in the former case is called conservative and the latter case is called dissipative. In this paper, we focus on the conservative solution case.

The continuation of the solutions after wave breaking has been studied extensively. There are different situations where such solutions are constructed; type of domain(periodic/non-periodic), vanishing/non-vanishing asymptotics in the case of non-periodic domain, energy preservation after the breakdown(conservative/dissipative), etc. We will list some previous known results on the global weak solutions by their distinct approaches. Bressan-Constantin \cite{BC1, BC2} and Holden-Raynaud \cite{HR1,HR2,HR3,HR4,HR5} obtained global weak solutions by reformulating the CH equation into a semilinear system of ODEs after introducing a new set of independent and dependent variables. Another approach was taken by Xin-Zhang \cite{XZ1} using the limit of viscous approximation. Bressan-Fonte \cite{BF} defined a Lipschitz distance functional to extract global weak solutions as the uniform limit of multi-peakon solutions. Grunert-Holden-Raynaud \cite{GHR1,GHR2} defined the new Lipschitz metric that is consistent with the construction of the solutions as in \cite{HR1} and \cite{HR4}. The authors also studied the aspects of global conservative solutions of the CH equation with nonvanishing asymptotics \cite{GHR3} and as a limit of vanishing density in the two-component CH system \cite{GHR4}.


In this paper, we construct global weak conservative solutions of the CH equation by using a simple change of variables on the Lagrangian variable $\eta$:
\begin{equation}\label{rho}
\rho=\sqrt{\eta_x}.
\end{equation}
This idea is motivated by Lenells \cite{ L2} who constructed global weak conservative solutions of the Hnter-Saxton(HS) equation
$$u_{txx}+2u_x u_{xx}+uu_{xxx}=0,\;\;t\ \in \RR,\;x \in S^1.$$
The author used the geometric interpretation that the HS equation describes the geodesic flow on the $L^2$ sphere via \eqref{rho}. In particular, the geodesic remains on the $L^2$ sphere for all time and this makes it possible to continue the geodesic in the weak sense after the blowup time. Since both CH and HS equations form singularities via wave breaking, it is natural to expect that the same kind of weak continuation holds true in the CH equation. We will show that the transformation \eqref{rho} removes the wave breaking singularities in the CH equation case as well, even though we don't have exactly the same geometric picture and explicit formula as in HS equation case.

In the geometric perspective, wave breaking is described by the particle trajectory $\eta$ forming a horizontal tangent as it evolves in time. This means that the geodesic flow hits the boundary of the diffeomorphism group. The CH equation written in $\rho$ variables has a global solution from which we can reconstruct the flow $\eta$ defined by
\begin{equation}\label{eta}
\eta(t,x):=\int_0^x \rho^2(t,y) dy+c(t),\;\text{where $c(t)$ is some function of time,}
\end{equation}
in the space of absolutely continuous functions. This flow is a weak geodesic for almost all time since the spatial derivative of $\eta$ vanishes on a set of measure zero for the most time. The idea is that $\rho$ can assume both signs and $\rho$ passes through the axis whenever it vanishes. The sign change of $\rho$ ensures that if $\rho$ vanishes at a point $x_0$ at some time $T$, $\rho(t,x_0)$, as a function of a time, does not vanish on the punctured neighborhood of $T$. It is certainly possible that $\rho$ vanishes at different places as it evolves, but the imporatant point is that the flow $\eta$ remains a homeomorphism for almost all time. This is precisely how the singularity of the CH equation is removed by introducing the new Lagrangian variable $\rho$. 

Our construction of the global weak solution shows that the spatial smoothness of the Lagrangian trajectories $\eta$ in \eqref{eta} is completely determined by the smoothness of $\rho$, which is dependent on the smoothness of the initial condition(see Proposition 3). This is an interesting phenomenon of the CH equation observed by McKean \cite{M2}; the solution of the CH equation experiences the jump discontinuity of its slope even for the smooth initial data, but the Lagrangian trajectory $\eta$ is spatially smooth for all time. We will prove that the Lagrangian flow $\eta$ is in $C^k$ for all time whenever the initial condition $u_0$ is in $C^k$. Our result improves the McKean's by showing the exact correspondence between the smoothness of the initial condition and the smoothness of the Lagrangian flow. Also, our approach does not use the complete integrability and explicit formula for the solutions of the CH equation. 

The outline of the paper is following. In Section 2, we write the CH equation in $\rho$ variables defined by \eqref{rho}. Then we will regard the resulting equation as an abstract ODE independent of the derivation and prove that solution exists globally. Using this global solution in $\rho$ variables, in Section 3 we will construct global weak conservative solutions of the original CH equation. In particular, we obtain the same global results for the CH equation, which is due to Bressan-Constantin \cite{BC1}, in much simpler way. Also, by using the estimates we already have, we can improve the result of McKean \cite{M2} on the persistence of the smoothness of Lagrangian trajectories. Finally, Section 4 contains some conclusions and remarks. The main theorems of this paper are the following:
\begin{theorem}
The Cauchy problem for the periodic Camassa-Holm equation
\begin{equation}\label{wCH}
\left\{
  \begin{array}{l l}
       u_t+uu_x=-(1-\partial_x^2)^{-1}\partial_x\left(u^2+\frac{u_x^2}{2}\right)\\
       u(0,x)=u_0(x)
   \end{array} \right.
\end{equation}
has a global solution $u \in C(\RR_+, H^1(S^1)) \cap \mathrm{Lip}(\RR_+,L^2(S^1))$. The solution is weak in the sense that the equality in the equation \eqref{wCH} is satisfied in the distributional sense. Also, the solution $u$ is conservative; $\norm{u(t)}_{H^1}=\norm{u_0}_{H^1}$ for $t \in \RR_+$ almost everywhere.
\end{theorem}

\begin{theorem}
Let $\eta(t,x)$ be the Lagrangian flow of the weak solution $u$ of the CH equation \eqref{wCH}. Then $\eta$ is spatially absolutely continuous for all time. Furthermore, if $\eta$ is initially $C^k$, then $\eta$ remains $C^k$ for all time.
\end{theorem}

\section{The Setup}

In terms of the diffeomorphism $\eta$ defined by the flow equation \eqref{flow}, the Cauchy problem \eqref{wCH} can be written in $(\eta,\eta_t)$ variables as following:
\begin{equation}\label{CH2}
\left\{ 
  \begin{array}{l l}
    \eta_{tt}=-\left\{\Lambda^{-1} \partial_x \left[\left(\eta_t \circ \eta^{-1}\right)^2+\frac{1}{2}\left(\frac{\eta_{tx}}{\eta_x}\circ \eta^{-1}\right)^2\right]\right\}\circ \eta\\
     \eta(0,x)=x,\;\;\eta_t(0,x)=u_0(x)
   \end{array} \right.
\end{equation}
Here, $\Lambda^{-1}=(1-\partial_x^2)^{-1}$ is the operator defined by
$$\Lambda^{-1}u(x)=\int_{S^1}g(x-y)u(y)dy,$$
where $g(x)=\frac{\cosh(|x|-1/2)}{2\sinh(1/2)}$. We want to write this equation \eqref{CH2} in terms of the new variable $\rho$ defined by \eqref{rho}. To do this, we use the conserved quantities of the CH equation to redefine quantities that appear in the equation. We will first assume that $\norm{\rho}_{L^2}=1$ to derive the equation. Later, we will prove the local existence and uniqueness of the solution for $\rho \in L^2(S^1)$ without the constraint on the norm. In order to show that the solution is global, however, we will restrict $\rho$ back to the unit sphere.

First, the mean of the velocity $u$ of the CH equation is conserved, so $\mu=\int_0^1 u(t,x) dx=\int_0^1 u_0(x) dx$ is a constant determined by the initial condition $u_0$. By changing variables, we have
$$\mu=\int_0^1 (u \circ \eta) \cdot \eta_x dx=\int_0^1 \eta_t \eta_x\;dx=\int_0^1 \eta_t \rho^2 dx.$$
So $\eta_t$ is a function of $\mu$ and $\rho$. By integrating $\eta_{tx}=2\rho\rho_t$ in the spatial variable $x$, we get $\eta_t=\int_0^x 2\rho\rho_t dy+c(t)$ for some function of time $c(t)$. We can substitute this expression into the $\mu$ equation, and determine $c(t)=\mu-\int_0^1 \int_0^y 2\rho\rho_t dz \rho^2 dy$. Hence, we have determined $\eta_t$ completely in terms of $\mu$, $\rho$, and $\rho_t$. Denoting $\eta_t$ in new variable $G$, we have
$$G(\mu,\rho,\rho_t)(t,x):=\int_0^x 2\rho\rho_t dy+\mu-\int_0^1 \int_0^y 2\rho\rho_t dz \rho^2 dy.$$
Next, we can write $\frac{\eta_{tx}^2}{2\eta_x}=2\rho_t^2$ since $\rho_t=\frac{\eta_{tx}}{2\sqrt{\eta_x}}$.

We are now ready to rewrite the equation \eqref{CH2} in $\rho$ variables. By differentiating the equation \eqref{CH2} with respect to $x$, we get
\begin{equation}\label{xxt}
\eta_{xtt}=\left(\eta_t^2\eta_x +\frac{\eta_{tx}^2}{2\eta_x}\right)-\Lambda^{-1}\left[\left(\eta_t \circ \eta^{-1}\right)^2+\frac{1}{2}\left(\frac{\eta_{tx}}{\eta_x} \circ \eta^{-1}\right)^2\right]\circ \eta \cdot \eta_x.
\end{equation}
Here, we used the identity $\Lambda^{-1}\left(-\partial_x^2\right)=1-\Lambda^{-1}$ and the chain rule. 
By using the explicit formula for $\Lambda^{-1}$ and changing variables in the integration, we have
\begin{align}
\Lambda^{-1}\left[\left(\eta_t \circ \eta^{-1}\right)^2+\frac{1}{2}\left(\frac{\eta_{tx}}{\eta_x}\circ \eta^{-1}\right)^2\right] \circ \eta=&\int_0^1 \frac{\cosh\left(|\eta(x)-y|-\frac{1}{2}\right)}{2\sinh(1/2)}\left[\left(\eta_t \circ \eta^{-1}\right)^2+\frac{1}{2}\left(\frac{\eta_{tx}}{\eta_x}\circ \eta^{-1}\right)^2\right]dy.\\
=&\int_0^1 \frac{\cosh\left(|\eta(x)-\eta(y)|-\frac{1}{2}\right)}{2\sinh(1/2)}\left(\eta_t^2\eta_y+\frac{\eta_{ty}^2}{2\eta_y}\right)dy.
\end{align}
Since $\eta(x)-\eta(y)=\int_y^x \eta_z dz=\int_y^s \rho^2 dz$, we can determine this integral completely in terms of $\rho$, $\rho_t$, and $G$. Denoting it in new variable $F$, we have
\begin{align}\label{F}
F(\mu, \rho,\rho_t)(t,x):=\int_0^1 \frac{\cosh\left(\left\vert\int_y^x \rho^2 dz\right\vert-\frac{1}{2}\right)}{2\sinh(1/2)}(\rho^2G^2+2\rho_t^2)dy
\end{align}
Since $\eta_{xtt}-\frac{\eta_{tx}^2}{2\eta_x}=2\rho\rho_{tt}$, the equation \eqref{xxt} becomes
$$2\rho\rho_{tt}=G^2 \rho^2-\rho^2 F.$$
Dividing by $2\rho$ on both sides, we get the following Cauchy problem in $\rho$ variables:
\begin{equation}\label{CH3}
\left\{ 
  \begin{array}{l l}
    \rho_{tt}=\frac{1}{2}\rho\left(G^2-F\right)\\
     \rho(0,x)=1(\text{constant function}),\;\;\rho_t(0,x)=\frac{1}{2}u_0'(x)
   \end{array} \right.
\end{equation}
Now, we want to solve this equation by viewing it as an abstract ODE in $(\rho, \rho_t)$ variables independent of the above derivation. That is, we assume that $\rho$ and $\rho_t$ are just functions in $L^2(S^1)$ satisfying the equation \eqref{CH3}. Note that this second order equation describes the integral curve of the vector field $(\rho,\rho_t) \mapsto f(\rho,\rho_t):=\frac{1}{2}\rho\left(G^2(\rho,\rho_t)-F(\rho,\rho_t)\right)$. We can easily show that $f$ is smooth in $(\rho,\rho_t)$ since $f$ is essentially a polynomial in $\rho$ and $\rho_t$, and this implies the local existence and uniqueness of the solution.

\begin{proposition}
The system
$$\left\{ 
  \begin{array}{l l}
     \frac{d \rho}{d t}= \rho_t\\
     \frac{d \rho_t}{d t}=f(\rho,\rho_t)
   \end{array} \right.$$
with initial conditions $\rho(0,\cdot)=1$(constant function) and $\rho_t(0,\cdot)=\frac{1}{2}u_0'$ describes the flow of a $C^{\infty}$ vector field on $TL^{2}(S^1)=L^2(S^1) \times L^2(S^1)$ and the curve $(\rho,\rho_t)$ exists for some time $T>0$.
\end{proposition}

\begin{proof}
We first compute the variational derivative in $\rho_t$.
$$\frac{\partial f}{\partial \rho_t}=\frac{1}{2}\rho\left(2G\frac{\partial G}{\partial \rho_t}-\frac{\partial F}{\partial \rho_t}\right).$$
Note that $G$ is linear in $\rho_t$ so it is smooth with respect to $\rho_t$. Also, for $\psi \in L^2(S^1)$,
$$\frac{\partial F}{\partial \rho_t}(\psi)=\int_0^1 \frac{\cosh\left(\left\vert \int_y^x \rho^2 dz\right\vert-\frac{1}{2}\right)}{2\sinh(1/2)}\left(2G\frac{\partial G}{\partial \rho_t}+2\rho_t \psi\right)dy,$$
and we can check that $\frac{\partial F}{\partial \rho_t}$ is in $L^2$ and continuous with respect to $\rho_t$. By repeating differentiation, we can see that $F$ is also smooth with respect to $\rho_t$. Similarly, we can show that $f$ is smooth with respect to $\rho$.
\end{proof}

In order to show that the smooth curve desribed by the equation \eqref{CH3} exists for all time, we need to restrict the base point $\rho$ to be on the unit sphere $U:=\{\rho \in L^2(S^1)\;:\;\norm{\rho}_{L^2}=1\}$. The resaon why we need this constraint is because $F$ and $G$ are periodic on $S^1$ only if $\int_0^1 \rho^2 dx=1$. Also, this condition is necessary for the weak geodesic flow for the CH equation defined spatially on $S^1$. We first prove that $\rho_t$ remains perpendicular to $\rho$ for all time when $\rho$ is constrained on the sphere.

\begin{proposition}
Let $f(\rho,\rho_t)$ be the vector field on $TL^2(S^1)$ defined by the equation \eqref{CH3}. If we restrict $\rho$ to be on the unit sphere $U$, then $f$ restricts on the tangent bundle $TU$.  
\end{proposition}
\begin{proof}
We want to show that $\int_0^1 \rho \rho_t dx=0$ for all time. When $t=0$, we have $\int_0^1 \rho(0,x)\rho_t(0,x)dx=\int_0^1 \frac{1}{2}u_0'(x)dx=0$ since $u_0$ is periodic. Next, we can compute
$$\frac{d}{dt}\int_0^1 \rho\rho_t dx=\int_0^1 \rho_t^2+\rho\rho_{tt}dx=\int_0^1 \rho_t^2+\frac{1}{2}\rho^2(G^2-F)dx.$$
We claim that $\int_0^1 \rho^2 F dx=2\rho_t^2+\rho^2 G^2$. Note that
\begin{align*}
\int_0^1 \rho^2 F dx=&\int_0^1 \rho^2 \int_0^x \frac{\cosh\left(\int_y^x \rho^2 dz-\frac{1}{2}\right)}{2\sinh(1/2)}(\rho^2 G^2+2\rho_t^2)dydx\\
&+\int_0^1 \rho^2 \int_x^1\frac{\cosh\left(-\int_y^x \rho^2 dz-\frac{1}{2}\right)}{2\sinh(1/2)}(\rho^2G^2+2\rho_t^2)dy dx\\
=&(I)+(II).
\end{align*}
By changing the order of integration, we have
\begin{align*}
(I)=&\int_0^1 (\rho^2 G^2+2\rho_t^2)\int_y^1 \rho^2 \frac{\cosh\left(\int_y^x \rho^2 dz-\frac{1}{2}\right)}{2\sinh(1/2)} dx dy\\
(II)=&\int_0^1 (\rho^2 G^2+2\rho_t^2)\int_0^y \rho^2 \frac{\cosh\left(-\int_y^x \rho^2 dz-\frac{1}{2}\right)}{2\sinh(1/2)} dx dy.
\end{align*}
Note that $\frac{d}{dx}\int_y^x \rho^2 dx=\rho^2$, so we can compute the inner integrals explicitly by FTC:
\begin{align*}
&\int_y^1 \rho^2 \frac{\cosh\left(\int_y^x \rho^2 dz-\frac{1}{2}\right)}{2\sinh(1/2)} dx+\int_0^y \rho^2 \frac{\cosh\left(-\int_y^x \rho^2 dz-\frac{1}{2}\right)}{2\sinh(1/2)} dx\\
=&\frac{1}{2\sinh(1/2)}\left\{\sinh\left(\int_y^x \rho^2 dz-\frac{1}{2}\right)\Big|_y^1-\sinh\left(-\int_y^x \rho^2 dz-\frac{1}{2}\right)\Big|_0^y\right\}\\
=&\frac{1}{2\sinh(1/2)}\left\{\sinh\left(\int_y^1 \rho^2 dz-\frac{1}{2}\right)-\sinh\left(-1/2\right)-\sinh\left(-1/2\right)+\sinh\left(-\int_y^0 \rho^2 dz-\frac{1}{2}\right)\right\}=1,
\end{align*}
since
$$\sinh\left(\int_y^1 \rho^2 dz-\frac{1}{2}\right)=\sinh\left(-\int_0^y \rho^2 dz+\frac{1}{2}\right)=\sinh\left(\int_y^0 \rho^2 dz+\frac{1}{2}\right).$$
Here, we used the assumption that $\int_0^1 \rho^2 dz=1$. Hence, $(I)+(II)=\int_0^1 (\rho^2 G^2+2\rho_t^2) dy$ and by substituting this back, we can conclude that $\frac{d}{dt}\int_0^1 \rho \rho_t dx=0$ for all time. Since the integral is initially zero, this shows that it remains to be zero for all time.
\end{proof}

Next, we want to show that the conservation of the $H^1$ energy in the original CH equation holds true in $\rho$ variables as well. This energy conservation will ensure the uniform boundedness of the estimates that we will need later. We first prove the following lemma.

\begin{lemma}
Let $G$ be defined as above. Then $G_t=-H$, where
$$H(t,x)=\int_0^x \frac{\sinh\left(\int_y^x \rho^2 dz-\frac{1}{2}\right)}{2\sinh(1/2)}(\rho^2G^2+2\rho_t^2)dy+\int_x^1\frac{\sinh\left(-\int_y^x \rho^2 dz-\frac{1}{2}\right)}{2\sinh(1/2)}(\rho^2G^2+2\rho_t^2)dy.$$
\end{lemma}

\begin{proof}
Note that
$$\frac{\partial H}{\partial x}=-(\rho^2G^2+2\rho_t^2)+\rho^2F.$$
By differentiating $G$ with respect to $t$, we have
\begin{align*}
G_t=&\int_0^x 2\rho_t^2+2\rho\rho_{tt} dy+c'(t)=\int_0^x 2\rho_t^2+\rho^2 G^2-\rho^2F dy+c'(t),
\end{align*}
where $c(t)=\mu-\int_0^1 \int_0^x 2\rho\rho_t\;dy \rho^2\;dx$. So
$$G_t=\int_0^x \frac{\partial H}{\partial y}dy+c'(t)=-H(t,x)+H(t,0)+c'(t).$$
So we want to show that $H(t,0)+c'(t)=0$. We can compute
\begin{align*}
c'(t)=&-\int_0^1 \int_0^x 2\rho_t^2+2\rho\rho_{tt} dy \rho^2 dx\\
=&-\int_0^1 \int_0^x 2\rho_t^2+\rho^2 G^2-\rho^2 F dy \rho^2 dx\\
=&-\int_0^1 (2\rho_t^2+\rho^2 G^2-\rho^2 F)\int_y^1 \rho^2 dx dy\;(\because \text{change the order of integration})\\
=&\int_0^1 \frac{\partial H}{\partial y}\int_y^1 \rho^2 dx dy\\
=&-H(t,0)+\int_0^1 H \rho^2 dy\\
=&-H(t,0),
\end{align*}
since $H\rho^2=F_x$ and $F$ is periodic. Hence, $G_t=-H$ as desired.
\end{proof}

\begin{proposition}
The `$H^1$ energy' of the $\rho$ equation is conserved:
\begin{equation}\label{conservation}
\frac{d}{dt}\int_0^1 \rho^2 G^2+4\rho_t^2\;dx=0.
\end{equation}
\end{proposition}
\begin{proof}
Recall that we have
\begin{center}
$G_x=2\rho\rho_t$ and $G_t \rho^2=-F_x$.
\end{center}
Then
\begin{align*}
\frac{d}{dt}\int_0^1 \rho^2 G^2+4\rho_t^2\;dx=&\int_0^1 2\rho\rho_t G^2+2\rho^2 GG_t+8\rho_t\rho_{tt}\;dx\\
=&\int_0^1 2\rho\rho_t G^2-2GF_x+4\rho\rho_t(G^2-F)\;dx\\
=&\int_0^1 6\rho\rho_t G^2\;dx-2G(1)F(1)+2G(0)F(0)\\
=&\int_0^1 \frac{d}{dx}\left[G^3\right] dx=0,
\end{align*}
since $F$ and $G$ are periodic.
\end{proof}

Now, we are ready to prove that the solution of the equation \eqref{CH3} is global when $\rho$ is restricted on the unit sphere. The idea is that all estimates are bounded in terms of $\norm{\rho}_{L^2}$ and $\norm{\rho_t}_{L^2}$ and restriction on the unit sphere guarantees that the two norms are uniformly bounded.

\begin{proposition}
The flow described by the equation \eqref{CH3}, when $\rho$ is restricted on the unite sphere, exists for all time.
\end{proposition}
\begin{proof}
We claim that the RHS of the equation \eqref{CH3} is uniformly bounded. Note that we have
\begin{equation}\label{estimate}
|G^2-F| \le |G^2|+|F|\le2 \norm{\rho}_{L^2}\norm{\rho_t}_{L^2}+\frac{1}{4\sinh(1/2)}\left(2\norm{\rho}_{L^2}^3 \norm{\rho_t}_{L^2}+\norm{\rho_t}_{L^2}^2\right).
\end{equation}
Since $\norm{\rho}_{L^2}=1$ and $\norm{\rho_t}_{L^2}$ is uniformly bounded by the energy conservation, $|G^2-F|$ is uniformly bounded as well. Then the RHS of the equation \eqref{CH3} is smooth and uniformly bounded in $(\rho,\rho_t)$. Thus, the solution of the Cauchy problem \eqref{CH3} can be extended for all time by Wintner's Theorem from the ODE theory(see \cite{H}).
\end{proof}

\section{Global Weak Solution of the CH equation in $u$ variable}

From the global solution $\rho$, we can readily construct global weak solutions for the original CH equation \eqref{CH2}. We first introduce new Lagrangian variables. Define
\begin{align}\label{Lagvar}
K(t,x):=&\int_0^x \rho^2 dy+tu_0(0)-\int_0^t \int_0^\tau H(s,0)\;ds d\tau,\\
G(t,x):=&\int_0^x 2\rho\rho_t dy+\mu-\int_0^1 \int_0^x 2\rho\rho_t dy \rho^2 dx.
\end{align}
As we have seen in the Lemma 5, $-H(s,0)=c'(t)$ where $c(t)=\mu-\int_0^1 \int_0^x 2\rho\rho_t dy \rho^2 dx$, and so $K$ satisfies the first order equation $\frac{\partial K}{\partial t}=G$. Next, we claim that $(K,G)$ solves the second order equation \eqref{CH2}. As in Lenells's paper \cite{L2}, we can decompose $S^1$ by $S^1=N \cup A \cup Z$ where
\begin{align*}
N:=&\{x \in S^1\;:\;\text{$K_x$ exists and equals 0, i.e., $\rho(t,x)=0$}\},\\
A:=&\{x \in S^1\;:\;\text{$K_x$ exists and $K_x(x)>0$, i.e., $\rho(t,x)>0$}\},\;\text{and}
\end{align*}
$Z$ is a set of measure zero. We first prove the following lemma.

\begin{lemma}
For almost all time $t \in \RR_+$, we have $\int_{N}\rho_t^2dy=0$.
\end{lemma}
\begin{proof}
It suffices to show that the set $N$ has a measure zero for almost all time. As in \cite{L2}, the Fubini theorem gives
$$\int_0^T m(N)dt=\int_{S^1}\int_0^T \chi_{\{\rho^{-1}(0)\}}dtdx,$$
for $T<\infty$ where $\chi_{\{\rho^{-1}(0)\}}:[0,\infty)\times S^1 \to \{0,1\}$ is the characteristic function. Hence, we want to show that the RHS of the equation vanishes.  Let $x_0 \in S^1$.Then the following set
$$N'(x_0)=\{0 \le t \le T\;:\;\rho(t,x_0)=0\}$$
has the Lebesgue measure zero. This is because $\rho_t(t,x_0)$, as a function of time $t$, cannot vanish on this set. If not, there is a time $t_0$ such that $\rho(t_0,x_0)=0=\rho_t(t_0,x_0)$. Note that we have the following differential inequality satisfied by the solution of \eqref{CH3}:
\begin{align*}
\frac{d}{dt}\left[\rho^2(t,x_0)+\rho_t^2(t,x_0)\right]=&2\rho\rho_t+2\rho_t\rho_{tt}\\
=&2\rho\rho_t+\rho\rho_t(G^2-F)\\
=&2\rho\rho_t\left(1+\frac{G^2-F}{2}\right)\\
\le&C(\rho^2+\rho_t^2),
\end{align*}
where $C=\max \left\{1+\frac{G^2-F}{2}\right\}$ is the uniform constant for the solution of the equation \eqref{CH3} restricted on the unit sphere(see the Proposition 7.) Hence, by the Gronwall's lemma, we have
$$\rho^2(t,x_0)+\rho_t^2(t,x_0) \le \left(\rho^2(t_0,x_0)+\rho_t^2(t_0,x_0)\right)e^{Ct}.$$
If $\rho(t_0,x_0)=0=\rho_t(t_0,x_0)$, the RHS vanishes and this implies that $\rho(t,x_0)=0=\rho_t(t,x_0)$ for all time $t$. In particular, $\rho(0,x_0)=0$ and this contradicts the initial condition of $\rho$. Hence, $\rho_t(t,x_0)$ must be nonzero on $N'(x_0)$. Since $\rho(t,x_0)$ is continuous as a function of a time, every neighborhood of $t_0$ contains a point where $\rho$ is nonzero, i.e., the set $N'(x_0)$ is isolated. Hence the set $N'(x_0)$ must be finite and it has a measure zero. 
\end{proof}
\begin{remark}
It is possible that $\rho$ vanishes on a set of positive measure, e.g., on an interval, at some time. However, the proof of the lemma shows that these appearances are rare and $\rho$ can vanish only on a set of measure zero for almost all time. Since $K_x=\rho^2$, the slope of the particle trajectories looks like a parabola. As a result, $K$ is generically a homeomorphism whenever $K_x$ vanishes on a set of measure zero. This implies that $K^{-1}$ is well-defined for almost all time.
\end{remark}

\begin{proposition}
For almost all time $t$, $(K, G)$ satisfies the CH equation
\begin{equation}\label{GKeqn}
G_t=-\left\{\partial_x \Lambda^{-1}\left[(G \circ K^{-1})^2+\frac{1}{2}\left(\frac{G_x}{K_x} \circ K^{-1}\right)^2\right]\right\}\circ K.
\end{equation}
Note that this solution is weak since $K$ is only absolutely continuous in the spatial variable $x$ and the equation is satisfied for almost all time $t$. 
\end{proposition}

\begin{proof}
We want to show that the equation $G_t=-H$ is equivalent to the equation \eqref{CH2} in the weak sense. Since $K$ is a diffeomorphism on $A$, we can change variables via $K$ and get
\begin{align*}
G_t=&-\int_{S^1} \frac{\sinh\left(\left\vert\int_y^x \rho^2 dz\right\vert-\frac{1}{2}\right)}{2\sinh(1/2)}(\rho^2G^2+2\rho_t^2)dy\\
=&-\int_{A}\frac{\sinh\left(\left\vert K(x)-K(y)\right\vert-\frac{1}{2}\right)}{2\sinh(1/2)}\left(K_x G^2+\frac{G_x^2}{2K_x}\right)dy+\int_{N}\rho_t^2dy\\
=&-\int_{K(A)}\frac{\sinh\left(\left\vert K(x)-y\right\vert-\frac{1}{2}\right)}{2\sinh(1/2)}\left(\left(G \circ K^{-1}\right)^2+\frac{1}{2}\left(\frac{G_x}{K_x} \circ K^{-1}\right)^2\right)dy+\int_{N}\rho_t^2dy\\
=&-\partial_x \left[\int_{K(A)}\frac{\cosh\left(\left\vert x-y\right\vert-\frac{1}{2}\right)}{2\sinh(1/2)}\left(\left(G \circ K^{-1}\right)^2+\frac{1}{2}\left(\frac{G_x}{K_x} \circ K^{-1}\right)^2\right)dy\right]\circ K+\int_{N}\rho_t^2dy.
\end{align*}
Since the Lebesgue measure of the set $K(A)$ is 1, the first integral on the RHS is equivalent to the RHS of the equation \eqref{CH2}. Also, $\int_{N}\rho_t^2dy=0$ for almost all time from the previous lemma. Thus, for almost all time $t \in \RR_+$, $(K,G)$ satisfies the equation \eqref{GKeqn} as desired.
\end{proof}

Now, we prove the main theorems of the paper. We first check that the velocity field $u$ satifying the flow equation \eqref{flow} is well-defined in $H^1$ for all time.

\begin{proposition}
Let $(K,G)$ be a weak solution of the CH equation \eqref{CH2} in the Lagrangian form. Then the velocity field $u \in C(\RR_+,H^1(S^1)) \cap \mathrm{Lip}(\RR_+,L^2(S^1))$ is well-defined by the formula
$$u(t,K(t,x))=G(t,x),\;\;(t,x) \in [0,\infty) \times S^1.$$
\end{proposition}
\begin{proof}
First, note that the flow $K$ is a bijection from $A \to K(A)$, so we can define $u$ a.e. on $S^1$ by
\begin{center}
$u(y):=G(K^{-1}(y))$ for $y \in K(A)$.
\end{center}
It remains to show that $u$ is well-defined by this formula on $N$, as well. The idea is that the Lagrangian velocity $G$ is well-defined for all time, so we can define the velocity $u$ from it.

Suppose that $K_x$ vanishes at some time $T$. Since $K$ is a nondecreasing absolutely continuous funciton, the set $N$ must be the union of isolated points and intervals. In the case where $x \in N$ is an isolated point, $K$ is still a bijection around a neighborhood of $x$ so we can defined the velocity $u$ in the same way as above. Suppose that $K_x$ vanishes on an interval $I \subset N$, i.e., $K_x(T,x)=0$ for $x \in I$. Then we have $K(T,x)=x^\ast$ for some $x^\ast \in S^1$ on $I$. Also, $\rho(T,x)=0$ on that interval and this implies that $G_x(T,x)=0$ since $G_x=2\rho\rho_t$ by differentiating in \eqref{Lagvar}. Then $G(T,x^\ast)$ is finite and is a constant on $I$, so we can define $u(T,x):=G(T,x^\ast)$ for $x \in I$. This completes the proof for $u$ being well-defined.

Once we define $u=G \circ K^{-1}$, we can show that its distributional derivative is $u_x=\frac{G_x}{K_x} \circ K^{-1}$. We omit the proofs for these and their regularity since it is exactaly the same argument as in \cite{L2}.
\end{proof}

\begin{remark}
When $K(T,x)=x^\ast$ on an interval $I$, we can think of this as a set of particles concentrating at one point $x^\ast$. This is consistent with the physical interpretation of the CH equation since it describes a compressible fluid motion. In this case, a set of particles concentrate at a point and moves in a same velocity.
\end{remark}

\subsection{Proof of Theorem 1}

We claim that $u$ is a weak solution of the equation \eqref{wCH}. We want to show that 
\begin{center}
${\displaystyle \int_{S^1 \times \RR_+}(u_t+u u_x) \phi\;dxdt=\int_{S^1 \times \RR_+}-p_x\phi\;dxdt}$ for all $\phi \in C_c^\infty(S^1 \times \RR_+)$,
\end{center}
where $p=\Lambda^{-1}(u^2+\frac{1}{2}u_x^2)$. We have
\begin{align*}
\int_{S^1 \times \RR_+}(u_t+u u_x) \phi\;dxdt=&\int_{S^1 \times \RR_+} -u\phi_t+uu_x\phi\;dxdt\\
=&\int_{S^1 \times \RR_+}-U K_x \phi_t \circ K+UU_x \phi \circ K\;dxdt,\;\text{where $U=u \circ K$}\\
=&\int_{S^1 \times \RR_+} U_t K_x \phi \circ K\;dxdt,
\end{align*}
since
\begin{align*}
\left(U K_x \phi \circ K\right)_t-\left(U^2 \phi \circ K\right)_x=&U_t K_x \phi \circ K-UU_x \phi \circ K+U K_x \phi_t \circ K,
\end{align*}
and $\phi$ has a compact support. Since $U_t=-\partial_x\Lambda^{-1}\left[u^2+\frac{1}{2}u_x^2\right] \circ K$, we get
$$\int_{S^1 \times \RR_+} U_t K_x \phi \circ K\;dxdt=\int_{S^1 \times \RR_+}-p_x\phi\;dxdt,$$
by another change of variables. Finally, the conservation of $H^1$ energy of the weak solution comes from rewriting all quantities appearing in \eqref{conservation} in $u$ variables. That is,
\begin{align*}
\int_{S^1} u^2+u_x^2 dx=&\int_{K(A)}\left(G \circ K^{-1}\right)^2+\left(\frac{G_x}{K_x}\circ K^{-1}\right)^2 dx\\
=&\int_{A} G^2 K_x+\frac{G_x^2}{K_x}dx\\
=&\int_{S^1} G^2 \rho^2+4\rho_t^2 dx
\end{align*}
is conserved for almost all time $t$. This completes the proof of the Theorem 1.

\subsection{Proof of Theorem 2}

We can observe that by providing extra smoothness on $\rho$ and $\rho_t$ variables, we can improve the smoothness of the $G$. In fact, the same estimate \eqref{estimate} will continue to work with $L^2$ norms replaced by $C^k$ norms since the spatial domain is compact. Hence, the solutions for the Cauchy problem \eqref{CH3} in $\rho$ variables will be global in $C^k$ spaces. Consequently, the Lagrangian variables $K$ and $G$ will be $C^{k+1}$ on $S^1$ for all time whenever $\rho$ and $\rho_t$ is $C^k$. This completes the proof of the Theorem 2.

\begin{remark}
Theorem 2 improves the result of McKean \cite{M2}. The difference between the current work and Mckean's work is that we don't need the assumption for the momentum $m=u-u_{xx}$ to satisfy $m, m_x \in H^1$ initially. Our result shows that that for each $k$, $K$ is exactly in $C^k$ whenever the initial condition $u_0$ is in $C^k$.
\end{remark}

\section{Future Research}

The global weak continuation of the CH equation in this research suggests that there might be a general theory that explains why it works. We can first suspect that the metrics corresponding to the HS and CH equations are close. That is, in the space of all Riemannian metrics on the group of diffeomorphisms, the Sobolev $H^1$ metric can be regarded as a nonlinear perturbation of $\dot{H}^1$ metric, where the global weak continuation property is a consequence of the robustness of the perturbation. However, the geometry of two diffeomorphism groups are different since the sectional curvature of $\dot{H}^1$ metric is a positive constant, whereas the $H^1$ metric has sectional curvature positive in `most directions' but also assumes negative sign(see \cite{L1, Mis1998})

\medskip

We can apply the change of variable technique in this paper to other Euler-Arnold equations or generalize this Lagrangian change of variable technique as well. In particular, we can consider the Wunsch equation which is the Euler-Arnold equation with a right invariant $\dot{H}^{1/2}$ metric:
\begin{equation}\label{wunsch}
\left\{
  \begin{array}{l l}
       \omega_t+u \omega_x+2\omega u_x=0,\\
       \omega=Hu_x,\;\text{$H$ is the Hilbert transform,}
   \end{array} \right.
\end{equation}
It was studied by Bauer-Kolev-Preston \cite{BKP} as a 1 dimensional vorticity equation for the 3 dimensional Euler's equation. The solution of Wunsch equation forms a finite time singularity along a particle trajectory due to wave breaking and the blow up result was further extended in the framework of Teichm\"{u}ller theory by Preston-Washabaugh \cite{PW}.
In this case, we can instroduce a new variable $Z=u_x+ i\omega$ to get
$$Z_t+uZ_x+Z^2=-F,$$
where $F$ is some positive function. In this framework, we can study the geodesic equation for the Sobolev $\dot{H}^s$ metric for $\frac{1}{2}<s<1$ in a uniform way. Since the sign change of the vorticity is the crucial assumption to get the blow up in all known cases, we can analyze the second-order ODE in this new Lagrangian variable $Z$ to find a direct proof of the blow up.

\medskip

Lastly, it would be interesting to interpret the global weak solution of the CH equation constructed in this paper in the context of optimal transport. Recently, Gallou\"{e}t-Vialard formulated generalized CH equation as an Euler-Arnold equation which can be identified as a particular solution of the incompressible Euler's equation on the group of homeomorphisms on $\RR^2 \setminus \{0\}$. We can investigate how the blow up and the global weak continuation of the CH equation can be described in the language of optimal transport, and understand the role of the change of variables \eqref{rho}.

\end{document}